\documentclass[12pt,english,spanish]{article}
\usepackage{amsmath}
\usepackage[T1]{fontenc}
\usepackage[letterpaper]{geometry}
\usepackage{amsmath}
\usepackage{amssymb}
\usepackage{comment}
\usepackage{graphicx}

\usepackage{epic}
\usepackage{pstricks}
\usepackage{pst-node}

\usepackage{pstricks-add}
\usepackage{amsfonts}
\usepackage{amsmath}
\usepackage{latexsym}
\usepackage[all]{xy}
\usepackage{amsthm}
\usepackage{graphicx}
\usepackage[T1]{fontenc}
\usepackage{hyperref}
\usepackage{caption}
\usepackage{verbatim}
\usepackage{amsmath}
\usepackage{amssymb}
\usepackage{amsthm}
\usepackage{amscd}
\usepackage{graphics}
\usepackage{graphpap}
\usepackage{float}

\newtheorem{definicion}{Definition}[section]
\newtheorem{nota}[definicion]{Remark}
\newtheorem{prop}[definicion]{Proposition}
\newtheorem{lema}[definicion]{Lemma}

\newtheorem{teorema}[definicion]{Theorem}
\newtheorem{cor}[definicion]{Corollary}
\newtheorem{ejp}[definicion]{Example}

\newcommand{\co}{\ensuremath{\colon}} 

\def\BBox{\kern  -0.2cm\hbox{\vrule width 0.2cm height 0.2cm}}

\renewcommand{\phi}{\varphi}

\newtheorem{theorem}{Theorem}[section]

\newtheorem{remark}[theorem]{Remark}

\title{Critical properties on Roman domination graphs}
\author{
A. Martínez-Pérez \thanks{Partially supported by MTM 2012-30719, alvaro.martinezperez@uclm.es},\\
{\small Universidad Castilla la Mancha, Spain.}\\
\\[1ex]
D. Oliveros\thanks{Supported by PAPIIT under project IN101912  and CONACyT project 166306,   dolivero@matem.unam.mx}\\
{\small  Instituto de Matem\'{a}ticas}\\
{\small  Universidad Nacional Aut\'{o}noma de M\'{e}xico, M\'{e}xico.}\\
}

\begin{document}
\maketitle

\begin{abstract}
A Roman domination function on a graph $G$ is a function $r:V(G) \to \{0,1,2\}$ satisfying the condition that every vertex $u$ for which $f(u)=0$ is adjacent
to at least one vertex $v$ for which $f(v)=2$. The weight of a Roman function is the value $r(V(G))= \sum_{u\in V(G)} r(u)$. The Roman domination number 
$\gamma_R(G)$ of $G$ is the minimum weight of a Roman domination function on $G$ . "Roman Criticality" has been defined in general as the study of graphs were
the Roman domination number decreases when removing an edge or a vertex of the graph. In this paper we give further results in this topic as well as the complete
characterization of critical graphs that have Roman Domination number $\gamma_R(G)=4$.
\end{abstract}





\section{Introduction}

According to \cite{Cockayne}, the Emperor Constantine the Great in the fourth century A.D. decreed that for the defense of the cities in the empire, any city without a 
legion stationed, must have a neighbor city having two stationed legions to secure it. A classical strategy problem, of course, is to minimize the total number of legions 
needed.  It becomes very natural to  generalize the problem to arbitrary graphs. Several studies and interesting results have been done in this topic, see for instance 
\cite{Cockayne} and \cite{Haynes} for basic properties on Roman domination functions, and for an excellent motivation on the topic see \cite{ReVelle} and \cite{steward}. 
Recently in \cite{Hansberg} the authors studied Roman criticality on Roman domination functions giving some results on critical block graphs and critical trees. 
Herein, we give some characterizations for some properties on vertex Roman critical graphs, edge Roman critical graphs and saturated Roman graphs with a slight difference on the definitions given in \cite{Hansberg}. Furthermore, we present a full characterization of these graphs when the minimal Roman number is four. 

\section{Characterizing Roman critical graphs}

Let $G=(V(G),E(G))$ be a simple graph of order $n$,  this is, a graph without loops and multiple edges and with $|V(G)|=n$. Let  $V(G)$ and $E(G)$ denote, as usual, the sets of vertices and edges respectively. The degree of a vertex $v$, $deg(v)$, is the number of edges in $E(G)$ incident to it. We denote by $N_{G}[v]$, or just $N[v]$ the closed neighborhood  of a vertex $v$ in $G$. As usual, $G\backslash \{v\}$ denotes the graph which is obtained by removing vertex $\{v\}$ together with all edges containing it.

\begin{definicion} A function  $r\co V(G)\to \{0,1,2\}$ is a  \emph{Roman domination function} if for every $u\in V(G)$ such that $r(u)=0$, then there is a vertex $v$ adjacent to $u$ so that $r(v)=2$.  The weight of a Roman domination function is the value $r(V(G))=\sum_{u\in V}r(u)$. The \emph{Roman domination number} of a a graph $G$, denoted by $\gamma_R(G)$ is  the minimum weight of all possible Roman domination functions on $G$. 

If a Roman function $r$ holds that $\gamma_R(X)=\sum_{v\in V(G)}r(v)$ we say that $r$ is a minimal Roman domination function.
\end{definicion}

Given a graph $G$ and a Roman domination function $r:V(G)\to \{0,1,2\}$ let $P:=(V_0;V_1;V_2)$  be the order partition of $V(G)$ induced by $r$, where $V_i=\{v\in V(G)\, |\, r(v)=i\}$ is called a
{\it Roman partition}. Clearly there is a one to one correspondence between Roman functions $r:V(G)\to \{0,1,2\}$ and Roman partitions $(V_0;V_1;V_2)$. Then, we may denote $r=(V_0;V_1;V_2)$.   

\noindent Observe that  If $H=(V(H),E(H))$ is a subgraph of $G=(V(G), E(G))$ with $V(H)=V(G)$ and $E(H)\subset E(G)$ then for any Roman domination function $r$ on $H$, $r$ is also a Roman domination function on $G$. Therefore $\gamma_R(G)\leq \gamma_R(H)$.

\begin{definicion} A graph $G$ is \emph{vertex critical} or \emph{v-critical} if $\gamma_R(G)=n$ and for every $v \in V(G)$, $\gamma_R(G\backslash \{v\})=n-1$.
\end{definicion}

\begin{definicion} A graph $G=(V(G),E(G))$ is \emph{Roman saturated} if for any pair of nonadjacent vertices $v,w$, the graph $G'=(V(G),E(G)\cup[v,w])$ holds 
that $\gamma_R(G')=\gamma_R(G)-1$
\end{definicion}

The previous definition was given in \cite{Hansberg} as $\gamma_R$-edge critical, however we believe that Roman saturated is a better name for this type of graphs.

\noindent The following useful lemma was proved in \cite{Hansberg} ,

\begin{lema}\label{Lemma: v-crit_0} A graph $G=(V(G),E(G))$ is v-critical if and only if  for every vertex $v$ there is a minimal Roman partition $(V_0;V_1;V_2)$ such that $v\in V_1$.
\end{lema}

\begin{definicion} A graph $G$ is \emph{nonelementary} if $\gamma_R(G)<|V(G)|$.
\end{definicion}

\paragraph{Cycles.}

Let $C_n$ by a cycle of length $n$. Then, $C_n$ is $v$-critical if and only if $n = 1, 2 \ mod(3)$. Moreover, if $n = 3k +1$, then $\gamma_R(C_n) = 2k +1$ and if $n = 3k +2$,
then $\gamma_R(C_n) = 2k + 2$. 

\noindent In particular, the cycle of length 5, $C_5$, is an example of a  nonelementary v-critical graph with $\gamma_R(G)=4$.

The following proposition gives a characterization of nonelementary graphs.

\begin{prop} A graph $G$ is nonelementary if and only if there is a connected component with at least 3 vertices. In particular, if $|V(G)|\geq 3$ and $G$ is connected, then it is nonelementary.
\end{prop}

\begin{proof} Let $C$ be a connected component of $G$ with at least 3 vertices. Then, there is a vertex $v$ with $deg(v)\geq 2$. Consider two edges $[v,w], [v,u]$. Thus, labelling $r(v)=2, r(w)=0, r(u)=0$ and $r(x)=1$ for every $x\neq u,v,w$ we obtain that $\gamma_R(G)<|V(G)|$.

If $G$ is nonelementary, then there is some Roman domination function $r$, some vertex $v$ with $r(v)=2$ and at least two vertices adjacent to $v$ labelled by 0. Therefore, the connected component containing $v$ has at least 3 vertices.
\end{proof}

\begin{remark} \label{R(G)3} $\gamma_R(G)\leq 3$ if and only if there exist a vertex $v\in V(G)$ with $deg(v)\geq n-2$
\end{remark}

 The following proposition was also shown in \cite{Hansberg}.

\begin{prop}\label{Prop: R-complete2} A graph $G=(V(G),E(G))$ is Roman saturated if and only if for every pair of vertices $v,w$ such that $[v,w]\not \in E(G)$ there exists a minimal Roman partition $(V_0;V_1;V_2)$ such that  $v\in V_1$ and $w\in V_2$ or $w\in V_1$ and $v\in V_2$.
\end{prop}

\begin{lema}\label{Prop: e-crit-n} Let $G=(V(G),E(G))$ be a v-critical graph with $\gamma_R(G)=n$. If $e\in E(G)$ and $G'=(V(G),E(G)\backslash \{e\})$, then $\gamma_R(G')=n$. 
\end{lema}

\begin{proof} We know that $\gamma_R(G')\geq \gamma_R(G)=n$. Consider any $e=[v,w]\in E(G)$. By Lemma \ref{Lemma: v-crit_0}, there is a minimal Roman partition of $G$, $P=(V_0;V_1;V_2)$, such that $v\in V_1$. Hence, $P$ is a Roman partition of $G'$ and $\gamma_R(G')\leq n$. Therefore, $\gamma_R(G')=n$.
\end{proof}

Hence, it makes sense to define $e$-critical graphs as follows: 

\begin{definicion} A v-critical graph $G=(V(G),E(G))$ is \emph{Roman critical on the edges} or \emph{e-critical} if for every edge $e\in E(G)$ the graph $G'=(V(G),E(G)\setminus \{e\})$ is not v-critical.
\end{definicion}
 

\begin{prop}\label{Prop: e-crit2} A v-critical graph $G=(V(G),E(G))$ is e-critical if and only if for every edge $e$ there exists a vertex $v_e$ such that for any minimal Roman partition $(V_0;V_1;V_2)$ with $v_e\in V_1$ then $e=[v,w]$ with $v\in V_0$ and $N[v]\cap V_2=\{w\}$.
\end{prop}

\begin{proof} If $G$ is e-critical and $e\in E(G)$, by \ref{Prop: e-crit-n}, $G\backslash \{e\}$ is not v-critical, this is, there exists $v_e$ such that $\gamma_R((G\backslash \{e\})\backslash \{v_e\})=n$ while $\gamma_R(G\backslash \{v_e\})=n-1$. Then, for any Roman partition of $G\backslash \{v_e\}$, $P=(V_0;V_1;V_2)$, with $\gamma_R(G\backslash \{v_e\})=n-1$ (in particular, for any Roman partition of $G$, $P'=(V'_0;V'_1;V'_2)$, with $\gamma_R(G)=n$ and $v_e\in V_1$), $(V_0;V_1;V_2)$ (resp. $(V'_0;V'_1\backslash \{v_e\};V'_2)$) is not Roman on $(G\backslash \{e\})\backslash \{v_e\}$ and, therefore, $e\in [V_0,V_2]$. Moreover, if $e=[v,w]$ with $v\in V_0$ the only vertex in $V_2$ adjacent to $v$ is $w$.

Now, consider any $e\in E(G)$ and let $v_e$ be such that for any minimal Roman partition $P=(V_0;V_1;V_2)$ with $v_e\in V_1$, then $e\in [V_0,V_2]$ and if $e=[v,w]$ with $v\in V_0$ then $w$ is the only vertex in $V_2$ which is adjacent to $v$. Hence, $P$ is not Roman for $G\backslash \{e\}$. Thus, there is no Roman partition $P=(V_0;V_1;V_2)$ on $G\backslash \{e\}$ with $v_e\in V_1$ and $\gamma(P)=n$. Hence, by Lemma \ref{Lemma: v-crit_0}, $G\backslash \{e\}$ is not v-critical. 
\end{proof}

\section{Roman domination number 4.}

It is immediate to check that the only elementary v-critical graphs with $\gamma_R(X)=4$ are $(V(G_1))=\{a,b,c,d\},E(G_1)=\emptyset)$, $(V(G_2)=\{a,b,c,d\},E(G_2)=\{[a,b]\})$ and 
$(V(G_3)=\{a,b,c,d\},E(G_3)=\{[a,b],[c,d]\})$.


\begin{lema}\label{lema: carac} Let $G$ be a nonelementary graph with $\gamma_R(G)=4$. Then, $G$ is v-critical if and only if for every $ x\in V(G)$ there exist two vertices $a_x,b_x$ so that $a_x\neq x\neq b_x$ and such that $N[a_x]=G\backslash \{x,b_x\}.$
\end{lema}

\begin{proof} The if part is clear. For every $x \in V(G)$, it suffices to define $r_x \colon V(G)\backslash\{x\} \to \{0,1,2\}$
so that $r_x(a_x) = 2$, $r_x(b_x) = 1$ and $r_x(y) = 0$  for all $y\neq a_x,b_x$. Then, $r_x$ is Roman and 
$\gamma_R(G\backslash \{x\})\leq 3$. (In fact, since $\gamma_R (G) = 4$, $\gamma_R(G\backslash \{x\}) = 3$).

Now let $G$ be a nonelementary v-critical graph with $\gamma_R(G) = 4$. Since it is
v-critical, for any vertex $x\in V(G)$ there is a Roman function $r_x \colon V(G)\backslash \{x\} \to \{0,1,2\}$
such that $\sum_{v\in V(G)\backslash \{x\}}  r_x(v) = 3$. Since it is nonelementary, $|V(G)\backslash \{x\}|>3$ and 
hence, there are two
vertices $y, z\in V(G)\backslash \{x\}$ so that $r_x(y) = 2$ and $r_x(z) = 1$. Since $\gamma_R(G \backslash \{x\})= 3$, $[y,z]\notin E(G)$.

Then, let $a_x = y$ and $b_x = z$. If $[a_x,x]\in E(G)$, then $r'_x \colon V(G) \to \{0,1,2\}$ such that $r'_x|_{V(G)\backslash \{x\}}:=r_x$ and $r'_x(x) = 0$ is Roman and $\gamma_R(G)\leq 3$ which is a contradiction. Therefore, $N[a_x]=G\backslash \{x,b_x\}.$
\end{proof}

\begin{cor}\label{Cor: carac} Let $G$ be a nonelementary v-critical graph of order $n$ and $\gamma_R(G)=4$ and consider $x,a_x,b_x$ as above. Then, either $N[x]=V(G)\backslash \{a_x,y\}$ or $N[b_x]=V(G)\backslash \{a_x,y\}$ for some $y\in V(G)$.
\end{cor}

\begin{proof} It suffices to apply again the proposition to $a_x$. There are edges from $a_x$ to every vertex different from $x,b_x$. Therefore, $a_{a_x}\in \{x,b_x\}$.
\end{proof}

Then, lemma \ref{lema: carac} implies the following useful Theorem.  

\begin{theorem}\label{teo: carac2} Let $G$ be a nonelementary graph of order $n$ and Roman domination number $\gamma_R(G)=4$. Then, $G$ is v-critical if and only if for every 
$x\in V(G)$ there exists a nonadjacent vertex $a_x$ with $deg(a_x)=n-3$. 
\end{theorem}

The following Proposition establishes that in a v-critical graph with Roman domination number four at least half of the vertices  have degree $n-3$.

\begin{prop} Let $G$ be a nonelementary v-critical graph of order $n$ and Roman domination number $\gamma_R(G)=4$. 
If $V_1:=\{ v \in V(G) \ | \ deg(v)=n-3\}$, then $|V_1|\geq \frac{n}{2}$.
\end{prop}

\begin{proof} By Theorem \ref{teo: carac2} we know that there is at least one vertex, $a_1$ such that $deg(a_1)=n-3$. Then, there are exactly two vertices $x_1,y_1$ so that $N[a_1]=V(G)\backslash \{x_1,y_1\}$. Consider any point $x_2\not \in \{x_1,y_1\}$. By the proof of  Theorem \ref{teo: carac2} there is some vertex $a_2$ such that $N[a_2]=V(G)\backslash \{x_2,y_2\}$ for some $y_2$ (not necessarilly different from $x_1,y_1$). Since $x_2\not \in \{x_1,y_1\}$, then $a_2\neq a_1$ and $|\{x_1,y_1,x_2,y_2\}|\leq 4$ . Then, we consider some vertex $x_3\not \in \{x_1,y_1,x_2,y_2\}$ and repeat the process until $\{x_1,y_1,...,x_k,y_k\}=V(G)$. Then $k\geq \frac{n}{2}$ and $a_1,...,a_k$ are $k$ different vertices with degree $n-3$.  
\end{proof}

\begin{nota} There is an infinite family of different nonelementary v-critical
graphs with $\gamma_R(G) = 4$.
\end{nota}

\begin{ejp}\label{Ejp: cicle} Consider a cicle, $C_n$, of length $n \geq 5$ were the vertices $V(C_n)$ are ordered in the natural way by
 $x_1,...,x_n$, and  $x_i,x_{i+1}\in E(C_n)$ for every $i=1,\dots ,n,  (mod \, n)$.  As we already know, if $n=5$ $\gamma_R(C_5)=4$ and it is critical. 
 Assume  $n \geq 6$ let us define $X_n$ by attaching some extra
edges to $C_n$ so that, for every vertex $x_i$, 
$N[x_i] = V(C_n)\backslash \{x_{i-2}, x_{i+2}\}$ (were the vertices are taken (mod \,  n))
See Figure \ref{Figure 1}. By Lemma \ref{lema: carac}, the resulting graph
$X_n$ is a nonelementary v-critical graph with $\gamma_R(X) = 4$.
\end{ejp}

\begin{figure}[ht]
\centering
\includegraphics[scale=0.4]{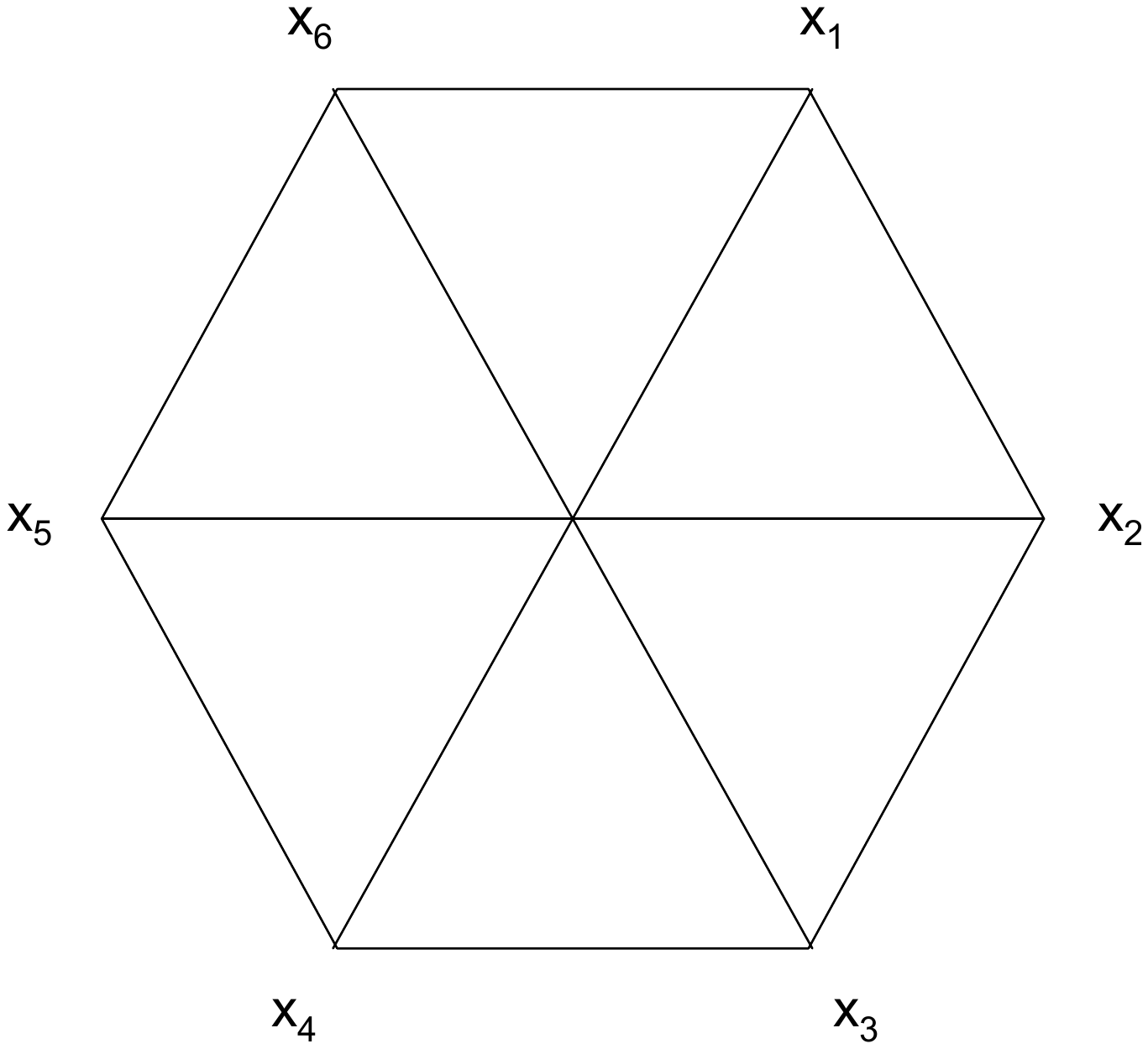}
\caption{$X_6$ is v-critical.}
\label{Figure 1}
\end{figure}

Observe that these are not the only examples of v-critical graphs with Roman domination number four. Take any of the $X_n$ above, with $n \geq 6$. Then,
if we remove any of those extra edges (imagine an hexagon with two
diagonals, $[x_1, x_4]$ and $[x_2, x_5]$, for example) it is still a nonelementary v-critical
graph with $\gamma_R(G) = 4$. 
Furthermore there are examples of nonelementary critical graphs with cut vertices. 

\begin{lema} If $G$ is a nonelementary v-critical graph with $\gamma_R(G) = 4$ and $v$ is a cut vertex,
then one of the connected components has a unique
vertex.
\end{lema}

\begin{proof} If $G\backslash \{v\}$ has two connected components
with at least two vertices, then $\gamma_R(G\backslash \{v\}) \geq 4$ and, therefore, $G$ is not v-critical.
\end{proof}

\begin{ejp} The graph represented on Figure \ref{Figure 2} has a cut vertex.


\begin{figure}[ht]
\centering
\includegraphics[scale=0.4]{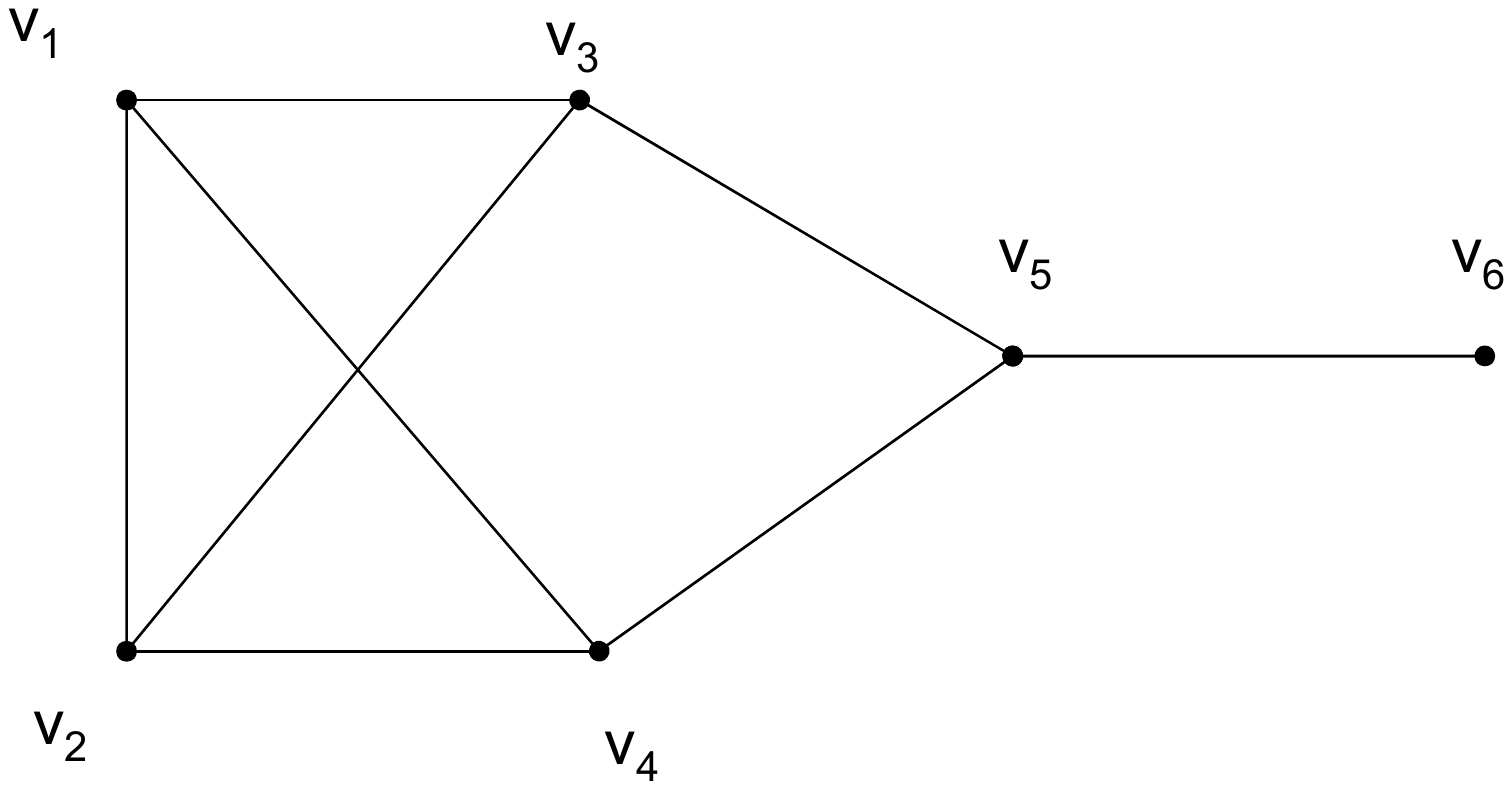}
\caption{A v-critical graph with a cut vertex $v_5$.}
\label{Figure 2}
\end{figure}
\end{ejp}

\begin{prop}\label{Prop: R-complete} A nonelementary graph $G=(V(G),E(G))$ of order $n$ and $\gamma_R(G)=4$ is Roman saturated if and only if any pair of vertices $v_1,v_2$ with $deg(v_i) < n-3$, $i=1,2$ then $[v_1,v_2]\in E(G)$. 
\end{prop}

\begin{proof} Suppose $v,w$ with $deg(v),deg(w) < n-3$ and $[v,w]\notin E(G)$. Let $G'=(V(G),E(G)\cup[v,w])$. 
Clearly, every vertex in $G'$ has degree less or equal than $n-3$ then by Remark  \ref{R(G)3}, $\gamma_R(G')>3$ and $G$ is not Roman saturated.
Assume that for any pair of vertices $v_1,v_2$ with $deg(v_i) < n-3$, $i=1,2$, $[v_1,v_2]\in E(G)$. Let  $G'=(V(G),E(G)\cup[v,w])$ for some $v,w\in V(G)$. Then 
one of them, say $v$, satisfies that $deg(v)\geq n-3$ in $G'$. Thus, again by Remark  \ref{R(G)3}, $\gamma_R(G)=3$.
\end{proof}

\begin{figure}[ht]
\centering
\includegraphics[scale=0.5]{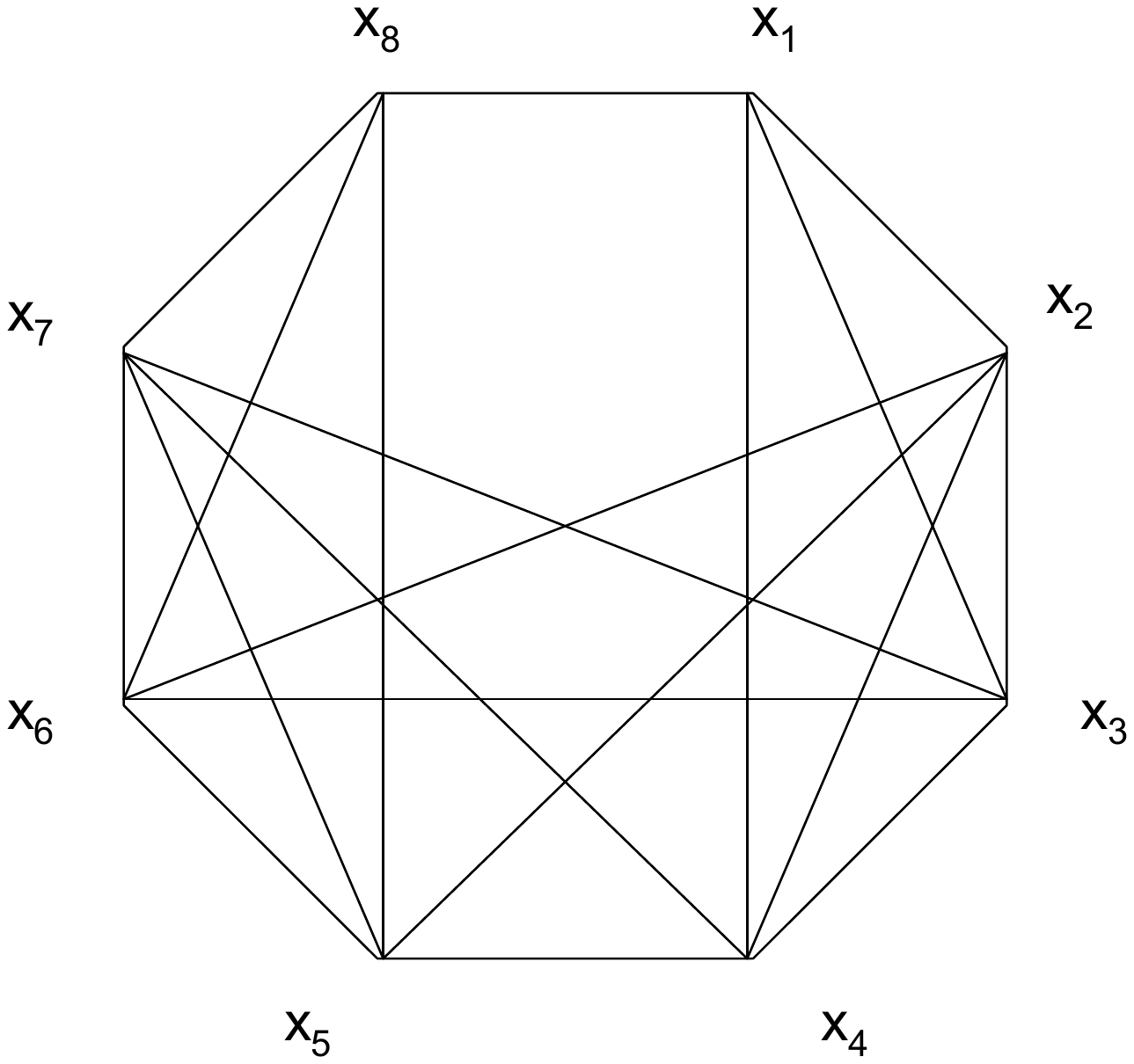}
\caption{A nonelementary v-critical Roman saturated graph with 8 vertices and two adjacent vertices of degree 4.}
\label{Octogono}
\end{figure}

\begin{lema} Let $G=(V(G),E(G))$ be a nonelementary Roman saturated v-critical graph of order $n$ and $\gamma_R(G)=4$. If $V_1:=\{v\in V(G) \ | \ deg(v)=n-3\}$, then $|V_1|\geq \frac{3n}{4}$.
\end{lema}

\begin{proof} Let $V(G)=\{v_1,v_2,\dots ,v_n\}$ and assume  $V_1=\{v_1,...,v_k\}$.  By  Proposition \ref{Prop: R-complete}, and Remark \ref{R(G)3} we know that for all  
$ v_r,v_s$ with $r,s>k$ then $[v_r,v_s]\in E(G)$. Therefore, since $deg(v_j)<n-3$ for every $j>k$, there are at least three vertices $v_{j_1},v_{j_2},v_{j_3}$ with $j_1,j_2,j_3\leq k$ such that $[v_{j_1},v_j],[v_{j_2},v_j],[v_{j_3},v_j]\not \in E(G)$. 

Also, by Corolary \ref{Cor: carac}, every $v_i$ with $i\leq k$ is joined at most to $k-2$ vertices in $V_1$. Hence, there is at most one vertex $v_j$ with $j>k$ such that $[v_{i},v_j]\not \in E(G)$. Therefore, given $v_j, v_{j_1},v_{j_2},v_{j_3}$ with the condition above, $v_{j_i}\neq v_{j'_{i'}}$ for any $j\neq j'$ or $i\neq i'$. Thus, $k\geq 3(n-k)$ and $k\geq \frac{3n}{4}$.
\end{proof}

\begin{nota} Notice that the graphs in the family described in Example \ref{Ejp: cicle} are Roman saturated since every vertex has degree $n-3$. Then, there is an infinite family of different nonelementary v-critical Roman saturated
graphs with $\gamma_R(G) = 4$. 
\end{nota}

Recall that $G$ is Roman critical on the edges or $e$-critical if for every edge $e$ the graph $G\setminus \{ e\}$ is not $v$-critical. 

\begin{prop}\label{Prop: e-crit} A nonelementary v-critical graph $G=(V(G),E(G))$ with $\gamma_R(G)=4$ is $e$-critical if and only if for every edge $e\in E(G)$ 
there is a vertex $v_e\in V(G)$ such that every vertex with degree $n-3$ nonadjacent to $v_e$ is a vertex contained in $e$.
\end{prop}

\begin{proof} 
%
%
%
Suppose $G$ is e-critical. Then, for every edge $e$ there is some $v_e$ such that $\gamma_R(G\setminus \{e\} \setminus  \{v_e\})=4$. Let $w$ be a vertex with $deg_G(w)=n-3$ and 
$[v_e,w]\not \in E(G)$. Then, $N[w]=G\setminus \{v_e,y\}$ for some vertex $y$. Let $r_w\colon V(G)\backslash \{v\}\to \{0,1,2\}$ such that $r_w(w)=2$, $r_w(y)=1$ and $r_w(x)=0$ for every $x\in V(G)\backslash \{v_e,w,y\}$. Since $\gamma_R(G\setminus \{e\} \setminus  \{v\})=4$, it follows that $r_w$ is not Roman on $G\setminus \{e\}$, this is, there is a vertex labelled with 0 which is not adjacent to $w$ on $G\setminus \{e\}$. Thus, for some $x\in V(G)\setminus \{v_e,y,w\}$ we obtain $e=[w,x]$.

Now suppose that for an edge $e\in E(G)$ there is a vertex $v_e\in V(G) $ such that every vertex $w$ with degree $n-3$ nonadjacent to $v_e$ is a vertex contained in $e$. 
Let us see that $G\setminus \{e\}$ is not v-critical. Let $r$ be any Roman function on $G\backslash \{e,v_e\}$ such that $\sum_{v\in V(G)\backslash \{v_e\}}r(v)=3$. Then, $r(w)=2$ for some vertex $w$ nonadjacent to $v_e$ with $deg_G(w)=n-3$. Let $N[w]=G\backslash \{v_e,y\}$ in $G$. It is also necessary that $r(y)=1$ since $y$ is not adjacent to $w$. Finally, by hypothesis, $e=[w,w']$ for some vertex $w'$ (obviously different from $v_e,y$). But then, since $r$ is a Roman function on $G\backslash \{e,v_e\}$, then $r(w')=1$ and 
$\sum_{v\in V(G)\backslash \{v_e\}}r(v)\geq 4$ leading to contradiction.
\end{proof}

\begin{nota} There is an infinite family of different v-critical and e-critical graphs. Consider the familly $\{X_n \ | \ n\geq 5\}$ described in Example \ref{Ejp: cicle}. For each $n$, consider $X_n=X_n^0,X_n^1,..., X_n^{k_n}$ a sequence of graphs where $X_n^i$ is obtained from $X_n^{i-1}$ by removing one edge so that $X_n^i$ is again v-critical until $X_n^{k_n}$ is e-critical (since $X_n$ is nonelementary, by Lemma \ref{Prop: e-crit-n}, there is always such a $k_n$). Since for each $n$ $|V(X_n)|=n$,  the graphs in $\{X_n^{k_n}\ | \ n\geq 5\}$ are all different from each other.
\end{nota}

\begin{prop}\label{Prop: cut} Let $G=(V(G),E(G))$ be a nonelementary v-critical, e-critical and Roman saturated (simple) graph with 
$\gamma_R(G)=4$. Let $V_1:=\{v\in V(G) \ | \ deg(v)=n-3\}=\{v_1,...,v_k\}$ and 
$V_2:=\{v\in V(G) \ | \ deg(v)<n-3\}=\{v_{k+1},...,v_n\}$. Then, either $G\cong C_5$ or $|V_2|=n-k=1$ and there is a cut vertex.
\end{prop}

\begin{proof} Claim: $|V_2|\leq 1$. Otherwise, suppose  there exist at least two vertices $v_r,v_s\in V_2$. By Proposition \ref{Prop: R-complete}, for all $v_r,v_s\in V_2$, $e=[v_r,v_s]\in E(G)$. By Proposition \ref{Prop: e-crit}, there is a vertex $v_e\in V(G)$ such that for every $v\in V_1$, either $v=v_r$, $v=v_s$ or $v$ is adjacent to $v_e$. In this case, since $v_r,v_s\in V_2$, every $v\in V_1\backslash \{v_e\}$ is adjacent to $v_e$. However, by Theorem \ref{teo: carac2}, there exists some vertex $a_{v_e}\neq v_e$ in $V_1$ such that $a_{v_e}$ is not adjacent to $v_e$ leading to contradiction. Therefore, $|V_2|\leq 1$.

Claim: If $|V_2|= 1$, then $deg(v_n)=1$ and there is a cut vertex. 
Suppose $n-k=1$ and $deg(v_n)\geq 2$. There is no loss of generality if we assume that 
$[v_{n-1},v_n],[v_{n-2},v_n]\in E(G)$. Applying Proposition \ref{Prop: e-crit} for $e_1=[v_{n-1},v_n]$ and $e_2=[v_{n-2},v_n]$ we know that there exist two vertices, 
$v_{e_1},v_{e_2}\in \{v_1,...,v_{n-3}\}$ such that for every $v\in V_1$, if $v\neq v_{e_1},v_{n-1}$ then $[v,v_{e_1}]\in E(G)$ and if 
$v\neq v_{e_2},v_{n-2}$ then $[v,v_{e_2}]\in E(G)$. See Figure \ref{Fig_cut}. Let us assume that $v_{e_1}=v_1$ and $v_{e_2}=v_2$. Hence, $N[v_{1}]=V(E)\backslash \{v_{n-1},v_n \}$ and 
$N[v_{2}]=V(E)\backslash \{v_{n-2},v_n \}$. (Notice that this implies that $v_{1}\neq v_{n-2},v_n$ and $v_{2}\neq v_{n-1},v_n$.) In particular, $e_3=[v_{1},v_{2}]\in E(G)$. Let us apply again Proposition \ref{Prop: e-crit} for $e_3$. Then, there is some vertex $v_{e_3}$ such that for every  $v\neq v_{1},v_{2},v_{e_3},v_n$, then $[v,v_{e_3}]\in E(G)$. Let us distinguish the following three cases:

\begin{itemize}
\item If $v_{e_3}=v_n$, then $v_n$ is adjacent to every vertex $v\neq v_1,v_2$ and $deg(v_n)=n-3$ which is a contradiction.

\item Let $[v_{e_3},v_n]\in E(G)$. Now, by the election of $v_{1}$, $v_{2}$, either $[v_{e_3},v_{1}]\in E(G)$ or 
$[v_{e_3},v_{2}]\in E(G)$. Therefore, $deg(v_{e_3})\geq 1+1+n-4=n-2$ which implies that $\gamma_R(G)=3$ leading to contradiction.

\item Let $v_{e_3}\neq v_n$ and $[v_{e_3},v_n]\not \in E(G)$. Then $v_{e_3}\neq v_{n-1}, v_{n-2}$ and, therefore, 
$[v_{e_3},v_1]\in E(G)$ and  $[v_{e_3},v_2]\in E(G)$. Thus, $v_{e_3}$ is adjacent to every vertex $v\neq v_n,v_{e_3}$ and $deg(v_{e_3})=n-2$ which implies that $\gamma_R(G)=3$ leading to contradiction.
\end{itemize}

Thus, there is a unique edge incident to $v_n$, let's say $[v_{n-1},v_{n}]$, and $v_{n-1}$ is a cut vertex.

\begin{figure}[ht]
\centering
\includegraphics[scale=0.5]{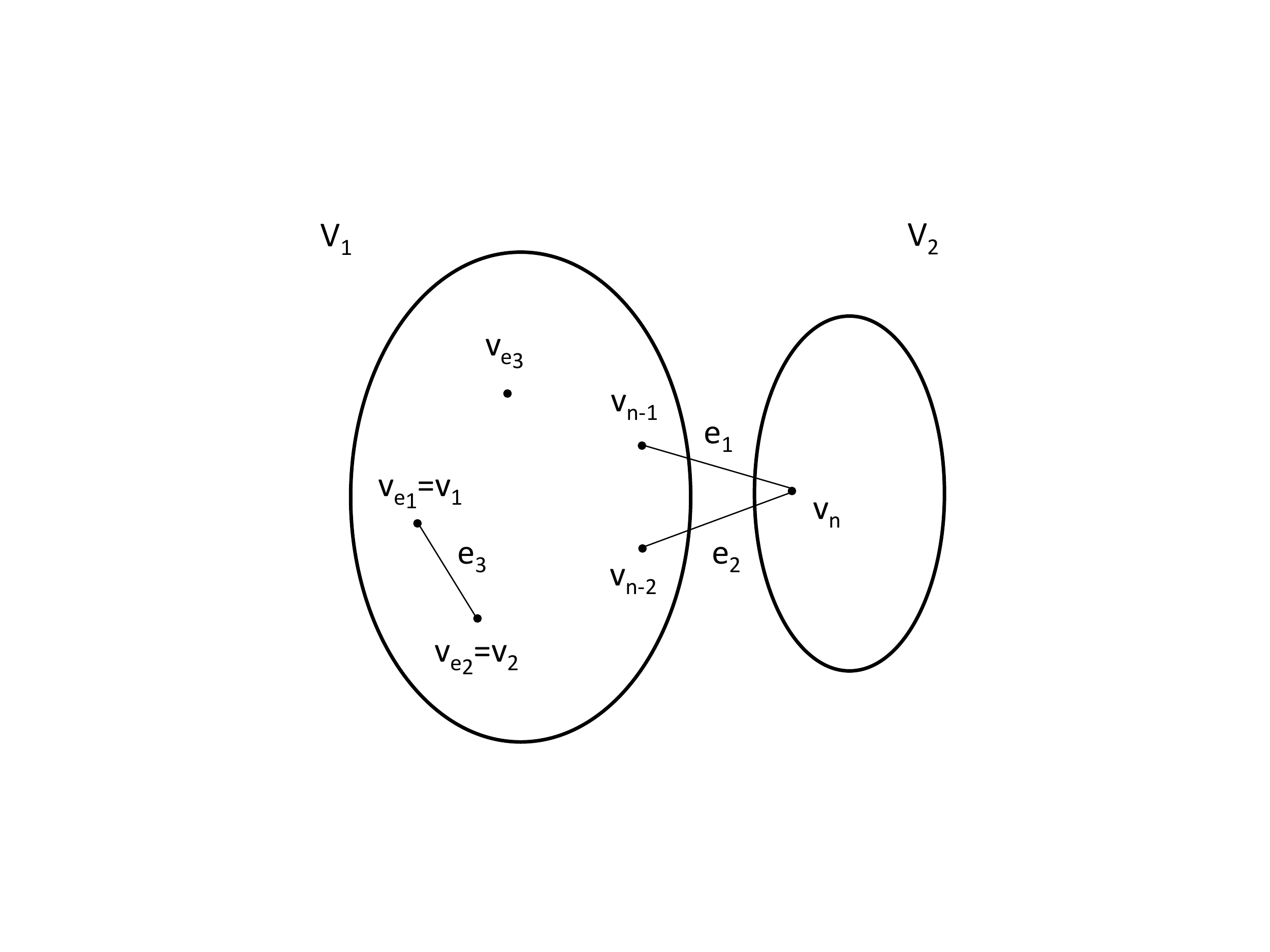}
\caption{If $|V_2|= 1$ and $deg(v_n)>1$, then there is a contradiction: $v_{e_3}\neq v_n$ and $deg(v_{e_3})=n-2$.}
\label{Fig_cut}
\end{figure}

Claim: If $|V_2|=0$, then $G\cong C_5$. 

Suppose $k=n$, this is, $deg(v)=n-3$ for every $v\in V(G)$. Consider all the edges incident to $v_1$: $e_1,...,e_{n-3}$ and let $e_i=[v_1,w_{i}],  i=1,n-3$). Then, by Proposition \ref{Prop: e-crit}, there are 
$n-3$ different vertices $v_{2},...,v_{n-2}$ such that $N[v_{i+1}]=G\backslash \{ v_1,w_i\}$ for $i=1,n-3$. 
In particular, $[v_1,v_{i+1}]\not \in E(G)$ for $i=1,n-3$. But since $deg(v_1)=n-3$, it follows that $n-3\leq 2$ and $n\leq 5$. 
Since we assumed that $G$ is nonelementary then $n=5$ and $deg(v_i)=2$ for every $1\leq i \leq n$. 

Let us assume with no loss of generality that 
$[v_1,v_2],[v_2,v_3]\in E(G)$. Then, if $[v_1,v_3]\in E(G)$ these vertices have already degree 2. Hence, $deg(v_4), deg(v_5)<2$ leading to contradiction. So, we may assume, again without loss of generality, that $[v_3,v_4]\in E(G)$.  Again, 
if $[v_1,v_4]\in E(G)$, then $deg(v_5)<2$. Therefore, $G\cong C_5$.
\end{proof}

\begin{cor}\label{Prop: C5} $C_5$ is the unique nonelementary (simple) graph, $G$, with $\gamma_R(G)=4$ which is v-critical, e-critical and Roman saturated with no cut vertices.
\end{cor}

\noindent For any even number $n\geq 6$, let us denote by $D_n=( V(D_n),E(D_n)$ a graph such that $ V(D_n)=\{v_1,...,v_n\}$ and 
$E(D_n)=\{[v_1,v_{2}]\}\cup\{[v_{n-1},v_{n}]\}\cup \{[v_j,v_{n-1}]:3\leq j \leq n-2\}\cup \{[v_r,v_j]: r=1,2  \mbox{ and } 3\leq j \leq n-2\}$$\cup \{[v_i,v_j]:3\leq i,j \leq n-2\}\backslash \{[v_{2j-1},v_{2j}]: 2\leq j \leq \frac{n-2}{2}\}$. 
See Figure \ref{Disconnecting}.

\begin{figure}[ht]
\centering
\includegraphics[scale=0.4]{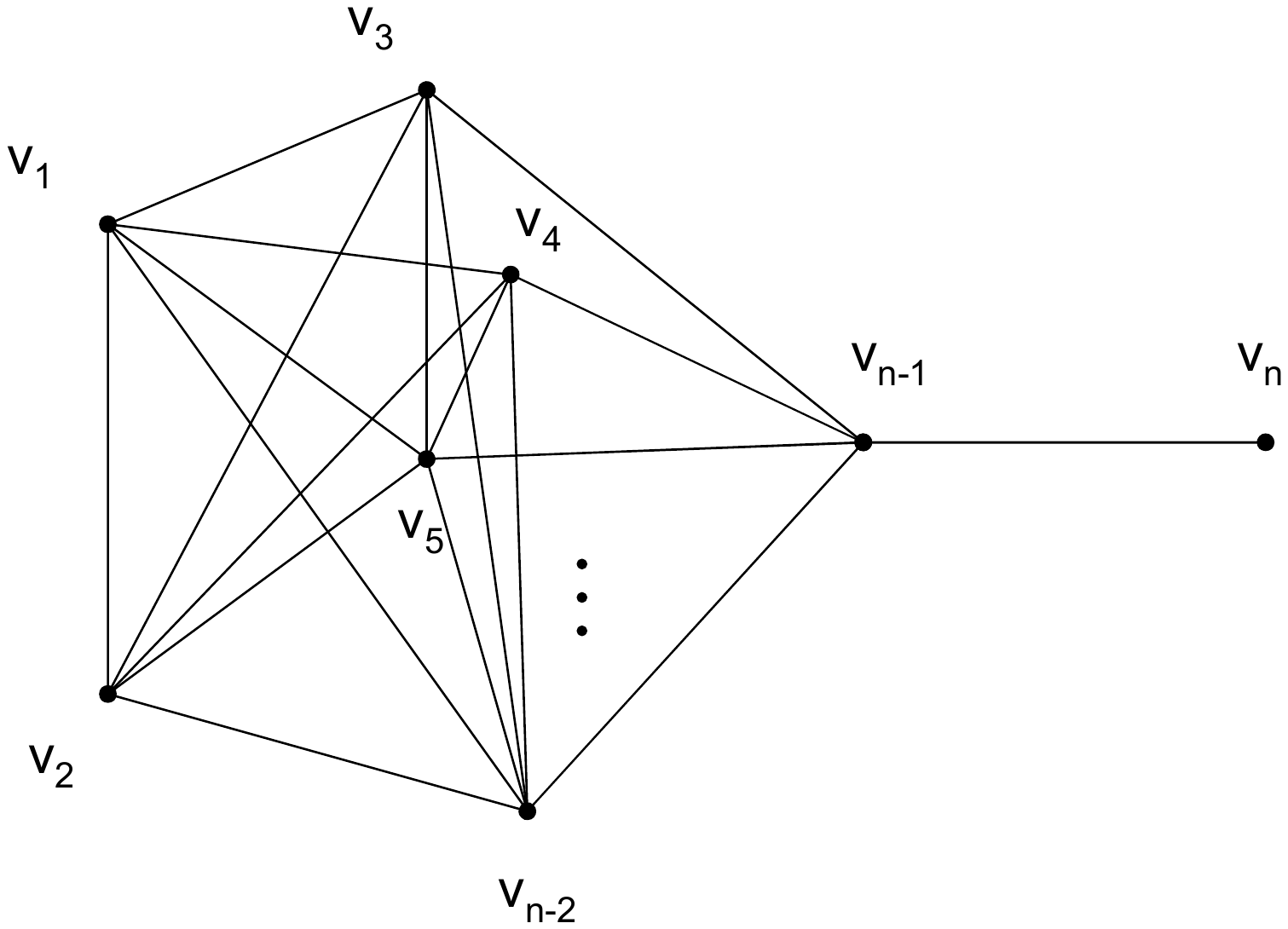}
\caption{If $n>5$ with $n$ even, then there is a unique v-critical, e-critical and Roman complete graph, $D_n$, with $\gamma_R(X)=4$.}
\label{Disconnecting}
\end{figure}

Notice that for every $1\leq i \leq n-1$, $deg(v_i)=n-3$ and $deg(v_n)=1$. In fact, if $r=1,2$, 
$N(v_r)=V(E)\backslash \{v_{n-1},v_n\}$, if $3\leq j \leq n-2$ with $j$ odd, then 
$N(v_j)=V(E)\backslash \{v_{j+1},v_n\}$, if $3\leq j \leq n-2$ with $j$ even, then 
$N(v_j)=V(E)\backslash \{v_{j-1},v_n\}$, $N(v_{n-1})=V(E)\backslash \{v_{1},v_2\}$ and $N(v_{n})=\{v_{n-1},v_n\}$.

\begin{prop} For every even number $n\geq 6$, $D_n$ is a nonelementary v-critical, e-critical and Roman saturated (simple) 
graph with $\gamma_R(G)=4$.
\end{prop}

\begin{proof}  Since $deg(v_i)=n-3$ for every $i\neq n$ it follows immediately that:

$G$ is v-critical by Theorem \ref{teo: carac2}; 

$G$ is Roman saturated.

To check that $G$ is e-critical, by Proposition \ref{Prop: e-crit}, it suffices to find for every edge $e=[x,y]$ a vertex $v_e$ which is adjacent to every vertex in  $\{v_1,...,v_{n-1}\}\backslash \{ x, y \}$. 

If $e=[v_{n-1},v_{n}]$, then $v_e=v_1$. 

If $e=[v_1,v_2]$, then $v_e=v_{n-1}$. 

For the rest of the edges, note that if  $v_{2j-1}\in e$, then $v_{2j}\notin e$ and $v_{2j}$ is adjacent to every vertex in $\{v_1,...,v_{n-1}\}\backslash\{v_{2j-1}\}$ and if $v_{2j}\in e$, then  $v_{2j-1}\notin e$ and $v_{2j-1}$ is adjacent to every vertex in $\{v_1,...,v_{n-1}\}\backslash\{v_{2j}\}$ 
for every $2\leq j\leq \frac{n-2}{2}$.  
\end{proof}

\begin{teorema}\label{Tma} If $G=(V(G),E(G))$ is a nonelementary v-critical, e-critical and Roman saturated (simple) graph with $\gamma_R(G)=4$, then:
\begin{itemize}
\item[a)] If $|V(G)|= 5$, then $G \cong C_5$

\item[b)] If $|V(G)|=n> 5$, then $n$ is even and, up to relabelling the vertices, $G \cong D_n$.
\end{itemize}
\end{teorema}

\begin{proof} In the case $|V_2|=0$ from the proof of Proposition \ref{Prop: C5}, $G=C_5$.

Now, let us see that the case $|V_2|= 1$ implies that $n>5$ is even and, up to relabelling the vertices, $G=D_n$. 

As we saw, there is a unique edge incident to $v_n$, $[v_n,v_{n-1}]$. Since $deg(v_{n-1})=n-3$ then there are two vertices, let's say $v_1,v_2$ such that $N[v_{n-1}]=G\backslash \{v_1,v_2\}$ and for every $3\leq j \leq n-2$, $[v_j,v_{n-1}]\in E(G)$. Since $deg(v_1)=deg(v_2)=n-3$, and $[v_r,v_k]\not \in E(G)$ for $r=1,2$ and $k=n,n-1$, it follows that $\{[v_r,v_j]:r=1,2 \mbox{ and } 3\leq j \leq n-2 \}\subset E(G)$. Finally, for every $v_i$, $deg(v_i)=n-3$ for every $3\leq i \leq k$. Therefore, for all $ 3\leq i\leq n-2$ there is exactly one $j\neq i$, $3\leq j \leq n-2$, such that $[v_i,v_j]\not \in E(G)$. In particular, we may assume, without loss of generality that these edges are $\{v_{2j-1},v_{2j}\}$ with $j=2,\frac{n-2}{2}$. Thus, we obtain the graph $D_n$.
\end{proof}

\begin{nota} There is an infinite family of v-critical, e-critical and Roman saturated (simple) graphs.
\end{nota}

In the following characterization we prove that for an arbitrary graph $G$ with $\gamma_R(G)=4$, the Roman criticality can be studied just by looking at small induced subgraphs. In fact, it suffices to check some properties on the induced subgraphs with 8 vertices to determine if the graph is $v$-critical, $e$-critical and Roman saturated. 

\begin{teorema}\label{Tma: 8} Let $G=(V(G),E(G))$ be a (simple) graph with $\gamma_R(G)=4$ and $|V(G)|=n\geq 8$. Then, $n$ is even and 
$G\cong D_n$ if and only if the following conditions hold: 

\begin{itemize}
\item[a)] There exist $\{ v_1,...,v_4\} \in V(G)$ such that $[v_1,v_i]\notin E(G)$ for every $2\leq i\leq 4$.

\item[b)] For every $\{ v_1,...,v_8\} \in V(G)$, if $[v_1,v_i]\notin E(G)$,  for all $ \, 2\leq i \leq 4$, then $|\{v_k:\ 1\leq k\leq 7 \mbox{ such that } [v_k,v_8]\in E(G) \}|\geq 5$.

\item[c)] For every $\{ v_1,...,v_6\} \in V(G)$, if $[v_1,v_i]\notin E(G)$ for every $2\leq i\leq 4$ then 
$|\{v_k: \ 5\leq k\leq 6 \mbox{ such that } [v_1,v_k]\in  E(G)  \}|\leq 1$.
\end{itemize}
\end{teorema}

\begin{proof}  Property $a)$ holds for some subset $\{v_1,...,v_4\}\in V(G)$ if and only if there is at least one 
vertex $v$ such that $deg(v)<n-3$.

Property $b)$ holds for every subset $\{v_1,...,v_8\}\in V(G)$ if and only if there is at most one vertex $v$ such that $deg(v)<n-3$.

Property $c)$ holds for every subset $\{v_1,...,v_6\}\in V(G)$ if and only if for every vertex $v$ with $deg(v)<n-3$ then $deg(v)=1$.

Finally, as we saw in the proof of Theorem \ref{Tma}, if $n> 5$ and the unique vertex $v$ such that $deg(v)<n-3$ holds that $deg(v)=1$, then $G\cong D_n$.
\end{proof}

Thus, from theorems \ref{Tma} and \ref{Tma: 8} we obtain the following:

\begin{cor} Let $G=(V(G), E(G))$ be a (simple) graph with $\gamma_R(G)=4$ and $|V(G)|=n\geq 8$. Then, $G$ is a v-critical, e-critical and Roman saturated (simple) graph if and only if the following conditions hold: 

\begin{itemize}
\item[a)] There exist $(v_1,...,v_4)\in V(G)$ such that $[v_1,v_i]\notin E(G)$ for every $2\leq i\leq 4$.

\item[b)] For every $(v_1,...,v_8)\in V(G)$, if $[v_1,v_i]\notin E(G)$, for every $ 2\leq i \leq 4$, then $|\{v_k:\ 1\leq k\leq 7 \mbox{ such that } [v_k,v_8]\in E(G) \}|\geq 5$. 

\item[c)] For every $(v_1,...,v_6)\in V(G)$, if $[v_1,v_i]\notin E(G)$ for every $2\leq i\leq 4$ then 
$|\{v_k: \ 5\leq k\leq 6 \mbox{ such that } [v_1,v_k]\in  E(G)  \}|\leq 1$.
\end{itemize}
\end{cor}



\begin{thebibliography}{99}

\bibitem{Cockayne} E.J. Cockayne, P.M. Dreyer Jr., S.M. Hedetniemi, and S.T. Hedetniemi, {\em{ On Roman domination in graphs}}, 
Discrete Math. 278 (2004), 11--22.

\bibitem{Hansberg} A. Hansberg, N. J. Rad., L. Volmann, {\em{Vertex and Edge critical Roman domination in graphs}} Util. Math. to appear.

\bibitem{Haynes} T.W. Haynes, S.T. Hedetniemi, and P.J. Slater, {\em{Fundamentals of Domination in Graphs}}, Marcel Dekker, NewYork, 1998.

\bibitem{ReVelle} C.S. ReVelle, K.E. Rosing, {\em{Defendens imperium romanum: a classical problem in military strategy}}, Amer Math. Monthly {\bf{107}} (2000),585--594.

\bibitem{steward} I. Stewart, {\em{ Defend the Roman Empire!}}, Sci. Amer. {\bf{281}} (6) (1999), 136--139.

\end{thebibliography}
\end{document}